\documentclass[11pt]{amsart}
\usepackage{amssymb}
\usepackage{amsmath}
\usepackage{amsfonts,latexsym,amstext}

\usepackage [T1]{fontenc}
\usepackage [latin1]{inputenc}
\usepackage{color}

\newtheorem{theorem}{Theorem}[section]
\newtheorem{definition}[theorem]{Definition}

\newtheorem{example}[theorem]{Example}
\newtheorem{lemma}[theorem]{Lemma}
\newtheorem{proposition}[theorem]{Proposition}
\newtheorem{corollary}[theorem]{Corollary}
\newtheorem{remark}[theorem]{Remark}

\def\u{\mathfrak{A}}
\def\s{\{\u_k\}_k}
\def\hbu{H_{b\u}}

\def\p{\mathcal{P}}
\def\v{\overset\vee}
\def\l{\mathcal{L}}

\def\bs{\backslash}

\newcommand{\TS[2]}{\bigotimes\nolimits^{#2,s}} 

\begin{document}


\title{Holomorphic Functions and polynomial ideals on Banach spaces}

\author{Daniel Carando}

\author{Ver\'{o}nica Dimant}
\author{Santiago Muro}


\thanks{Partially supported by ANPCyT PICT 05 17-33042. The first and third authors were also partially supported by UBACyT Grant X038 and ANPCyT PICT 06
00587.}

\address{Departamento de Matem\'{a}tica - Pab I,
Facultad de Cs. Exactas y Naturales, Universidad de Buenos Aires,
(1428) Buenos Aires, Argentina and CONICET} \email{dcarando@dm.uba.ar}

\address{Departamento de Matem\'{a}tica, Universidad de San
Andr\'{e}s, Vito Dumas 284, (B1644BID) Victoria, Buenos Aires,
Argentina and CONICET} \email{vero@udesa.edu.ar}

\address{Departamento de Matem\'{a}tica - Pab I,
Facultad de Cs. Exactas y Naturales, Universidad de Buenos Aires,
(1428) Buenos Aires, Argentina and CONICET} \email{smuro@dm.uba.ar}


\begin{abstract}
Given $\u$ a multiplicative sequence of polynomial ideals, we
consider the associated algebra of holomorphic functions of bounded
type, $H_{b\u}(E)$. We prove that, under very natural conditions
satisfied by many usual classes of polynomials, the spectrum
$M_{b\u}(E)$ of this algebra ``behaves'' like the classical case of
$M_{b}(E)$ (the spectrum of $H_b(E)$, the algebra of bounded type holomorphic functions). More precisely, we prove that
$M_{b\u}(E)$ can be endowed with a structure of  Riemann domain over
$E''$ and that the extension of each $f\in H_{b\u}(E)$ to the
spectrum is an $\u$-holomorphic function of bounded type in each
connected component. We also prove a Banach-Stone type theorem for
these algebras.
\end{abstract}

\maketitle

\section*{Introduction}
We consider algebras of analytic functions associated to sequences
of polynomial ideals. More precisely, if $\u=\{\u_k\}_k$ is a
coherent sequence of Banach polynomial ideals (see the definitions
below) and $E$ is a Banach space, we consider the space $\hbu(E)$ of
holomorphic functions of bounded type associated to $\u(E)$, much in
the spirit of holomorphy types introduced by Nachbin \cite{Nac69}
(see also \cite{Din71(holomorphy-types)}). These spaces of
$\u$-holomorphic functions of bounded type were introduced
in~\cite{CarDimMur07}, and many particular classes of holomorphic
functions appearing in the literature are obtained by this procedure
from usual sequences of polynomial ideals (such as nuclear,
integral, approximable, weakly continuous on bounded sets,
extendible, etc.). Under certain multiplicativity conditions on the
sequence of polynomial ideals (satisfied for all the mentioned
examples), $\hbu(E)$ turns out to be an algebra.

We study this multiplicativity condition and relate it with
properties of the associated tensor norms. We show that polynomial
ideals associated to natural symmetric tensor norms (in the sense
of~\cite{CarGal-4}) are multiplicative. We also prove that composition
and maximal/minimal hulls of multiplicative sequences of polynomial
ideals are multiplicative.

Whenever $\hbu(E)$ is an algebra, we study its spectrum $M_{b\u}(E)$ and show, under fairly general assumptions, that it has an analytic structure as a Riemann domain over $E''$, a result analogous to that for the spectrum $M_b(E)$ of $H_b(E)$, the algebra of all holomorphic functions of bounded type (see \cite{AroGalGarMae96} and \cite[Section 6.3]{Din99}). Moreover, the connected components of $M_{b\u}(E)$ are analytic copies of $E''$, and one may wonder if the Gelfand extension of a function $f$ to $M_{b\u}(E)$ is analytic and, also, if the restriction of this extension to each connected component can be thought as a function in $\hbu(E'')$. To answer these questions, we study the Aron-Berner extension of functions in $\hbu(E)$ and also translation and convolution operators on these kinds of algebras. We obtain conditions on the sequence of polynomial ideals that ensure a positive answer to both questions, which are satisfied by most  of the examples considered throughout the article.

Finally, we address a Banach-Stone type question on these algebras:
if $\hbu(E)$ and $H_{b\mathfrak B}(F)$ are (topologically and
algebraically) isomorphic, what can we say about $E$ and $F$? We
obtain results in this direction which allow us to show, for
example, that if $E$ or $F$ is reflexive and $\u$ and $\mathfrak B$
are any of the sequence of nuclear, integral, approximable or
extendible polynomials, then if $H_{b\u}(E)$ is isomorphic to
 $H_{b\mathfrak B}(F)$ it follows that $E$ and $F$  are isomorphic.

\medskip We refer to \cite{Din99,Muj86} for notation and results regarding
polynomials and holomorphic functions in general, to
\cite{Flo01,Flo02} for polynomial ideals and to \cite{DefFlo93,Flo97} for tensor products of Banach spaces.

\section{Preliminaries}

Throughout this paper $E$ will denote a complex Banach space and
 $\p^k(E)$ is the Banach space of all continuous
$k$-homogeneous polynomials from $E$ to $\mathbb{C}$. If $P\in
\p^k(E)$, there exists a unique symmetric $n$-linear mapping $\v
P\colon\underbrace{E\times\cdots\times E}_k\to \mathbb{C}$ such that
$$P(x)=\v P(x,\dots,x).$$

We define, for each $a\in E$,  $P_{a^j}\in \p^{k-j}(E)$ by $$
P_{a^j}(x)=\v P(a^j,x^{k-j})=\v
P(\underbrace{a,...,a}_j,\underbrace{x,...,x}_{k-j}).$$ For $j=1$,
we write $P_a$ instead of $P_{a^1}$.

\bigskip

Let us recall the definition of polynomial ideals
\cite{Flo01,Flo02}. A \textbf{Banach ideal of
(scalar-valued) continuous $k$-homo\-geneous polynomials} is a pair
$(\mathfrak{A}_k,\|\cdot\|_{\mathfrak A_k})$ such that:
\begin{enumerate}
\item[(i)] For every Banach space $E$, $\mathfrak{A}_k(E)=\mathfrak A_k\cap \mathcal
P^k(E)$ is a linear subspace of $\p^k(E)$ and $\|\cdot\|_{\u_k(E)}$
is a norm on it. Moreover, $(\u_k(E), \|\cdot\|_{\u_k(E)})$ is a
Banach space.

\item[(ii)] If $T\in \l (E_1,E)$ and $P \in \u_k(E)$, then $P\circ T\in \u_k(E_1)$ with $$ \|
P\circ T\|_{\u_k(E_1)}\le \|P\|_{\u_k(E)} \| T\|^k.$$

\item[(iii)] $z\mapsto z^k$ belongs to $\u_k(\mathbb C)$
and has norm 1.
\end{enumerate}

In \cite{CarDimMur09} we defined and studied coherent sequences of
polynomial ideals. Here we present more results about this topic.
Even though the original definitions were for general vector valued
polynomial ideals, we focus in this article on the case of scalar
valued polynomial ideals.

We recall the definitions:

\begin{definition}\label{deficoherente}
Consider the sequence $\u=\{\u_k\}_{k=1}^\infty$, where for each
$k$, $\u_k$ is a Banach ideal of scalar valued $k$-homogeneous
polynomials. We say that $\{\u_k\}_k$ is a \textbf{coherent sequence
of polynomial ideals} if there exist positive constants $C$ and $D$
such that for every Banach
space $E$, the following conditions hold for every $k\in\mathbb{N}$: \\
\begin{itemize}
\item[(i)]
For each $P\in \u_{k+1}(E)$ and $a\in E$, $ P_a$ belongs to
$\u_k(E)$ and
$$\|P_a\|_{\u_{k}(E)} \le C
\|P\|_{\u_{k+1}(E)} \|a\|$$
\item[(ii)]
For each $P\in \u_k(E)$ and $\gamma\in E'$, $\gamma P$ belongs to
$\u_{k+1}(E)$ and
$$\|\gamma P\|_{\u_{k+1}(E)}\le D \|\gamma\|
\|P\|_{\u_k(E)}$$
\end{itemize}
\end{definition}

In \cite{CarDimMur09}, many examples of coherent sequences were
presented (see Examples below). Let us see now that from any pair of
such examples, it is ``easy'' to construct many others.

If $\u_k^0$ and $\u_k^1$ are Banach ideals of $k$-homogeneous
polynomials, for each $0<\theta<1$,  we denote by $\u_k^{\theta}$
the polynomial ideal defined by
$$
\u_k^{\theta}(E)=\big[\u_k^{0}(E),\u_k^{1}(E)\big]_{\theta}\quad
\textrm{for every Banach space }E.
$$
That is,  $\big[\u_k^{0}(E),\u_k^{1}(E)\big]_{\theta}$ is the space
obtained by complex interpolation from the pair
$\big(\u_k^{0}(E),\u_k^{1}(E)\big)$ with parameter $\theta$. The
complex interpolation can be defined because both spaces are
included in the space of continuous $k$-homogeneous polynomials
$\p^k(E)$. Note also that $\u_k^{\theta}$ is actually a Banach ideal
by the properties of interpolations spaces, since the ideal
properties can be rephrased as the continuity of certain linear
operators.

The construction of interpolating spaces easily implies the
following result.

\begin{proposition}
 Let $\{\u_k^0\}_k$ and $\{\u_k^1\}_k$ be coherent sequences
of polynomial ideals with constants $C_0$, $D_0$ and $C_1$, $D_1$,
respectively. Then, for every $0<\theta<1$, the sequence
$\{\u_k^{\theta}\}_k$ is coherent with constants
$C_0^{1-\theta}C_1^{\theta}$ and $D_0^{1-\theta}D_1^{\theta}$.

\end{proposition}

\bigskip

There is a natural way of building a class of holomorphic functions associated to a coherent sequence of polynomial ideals. In \cite{CarDimMur07} we defined:

\begin{definition}
Let $\u=\s$ be a coherent sequence of polynomial ideals and $E$ be a
Banach space. We define the space of $\u$-holomorphic functions of
bounded type by
$$
\hbu(E)=\left\{f\in H(E)\ : \frac{d^kf(0)}{k!} \in \u_k(E) \textrm{
and } \lim_{k\rightarrow \infty}
\Big\|\frac{d^kf(0)}{k!}\Big\|_{\u_k(E)}^{1/k}=0 \right\}.
$$
\end{definition}

We define in $\hbu(E)$ the seminorms $p_R$, for $R>0$, by
$$
p_R(f)=\sum_{k=0}^{\infty}
\Big\|\frac{d^kf(0)}{k!}\Big\|_{\u_k(E)}R^k,
$$
for $f\in \hbu(E)$. It is proved in \cite{CarDimMur07} that this is a Fr\'{e}chet space. Since for every polynomial $P\in \u_k(E)$ we have $\|P\|\le \|P\|_{\u_k(E)} $, the space $\hbu(E)$ is continuously contained in $H_b(E)$.

The following examples of spaces of holomorphic functions of bounded type were already defined in the literature and can be seen as particular cases of the above definition.

\begin{example} \rm
\begin{enumerate}
\item[(a)] Let $\u$ be the sequence of continuous homogeneous polynomials
$\u_k=\mathcal P^k$, $k\geq 1$. Then $\hbu(E)=H_b(E)$.
\item[(b)] If $\u$ is the sequence of weakly continuous on bounded sets polynomial ideals then $\hbu(E)$
is the space of weakly uniformly continuous holomorphic functions of
bounded type $H_{bw}(E)$ defined by Aron in \cite{Aro79}.
\item[(c)] If $\u$ is the sequence of nuclear polynomial ideals then $\hbu(E)$
is the space of nuclearly entire functions of bounded type
$H_{Nb}(E)$ defined by Gupta and Nachbin (see \cite{Din99,Gup70}).
\item[(d)] If $\u$ is the sequence of extendible polynomials,
$\u_k=\mathcal P_e^k$, $k\ge 1$. Then, by~\cite[Proposition
14]{Car01},  $ \hbu(E)$ is the space of all $f\in H(E)$ such that,
for any Banach space $G\supset E $, there is an extension $ \tilde
f\in H_b(G)$ of~$f$.
\item[(e)] Let $\u$ be the sequence of integral polynomials,
$\u_k=\mathcal P_I^k$, $k\ge 1$. Then $\hbu(E)$ is the space of
integral holomorphic functions of bounded type $H_{bI}(E)$ defined
in \cite{DimGalMaeZal04}.
\end{enumerate}
\end{example}

\section{Multiplicative sequences}

\begin{definition}
Let $\s$ be a sequence of scalar valued polynomial ideals. We will
say that $\s$ is {\bf multiplicative} if it is coherent and there exists a constant $M$ such that for
each $P\in\u_k(E)$ and $Q\in\u_l(E)$, we have that
$PQ\in\u_{k+l}(E)$ and
$$
\|PQ\|_{\u_{k+l}(E)}\le M^{k+l}\|P\|_{\u_k(E)}\|Q\|_{\u_l(E)}.
$$
\end{definition}

Our interest in multiplicative sequences of polynomial ideals is
motivated by the following result from \cite{CarDimMur07}:

\begin{lemma}
Let $\u$ be a multiplicative sequence. If $f,g\in \hbu(E)$ then
$f\cdot g\in \hbu(E)$. Therefore, $\hbu(E)$ is an algebra.
\end{lemma}

In the following section we study the spectrum of this algebra.
Now, let us see some examples of multiplicative sequences.
\begin{example}\rm
\begin{enumerate}
\item[(a)] If $\u_k$ is the ideal of all $k$-homogeneous (or of
approximable, extendible, weakly continuous on bounded sets)
polynomials then $\s$ is a multiplicative
sequence.

\item[(b)]If $\u_k$ is the ideal of all $k$-homogeneous nuclear
polynomials then $\s$ is a multiplicative sequence (see for example
\cite[Exercise 2.63]{Din99} or deduced it as a consequence of the
following example and Corollary \ref{min y max}).

\item[(c)]If $\u_k$ is the ideal of all $k$-homogeneous integral
polynomials then $\s$ is a multiplicative sequence. Indeed, for
$P\in\p_I(^kE)$, $Q\in\p_I(^lE)$, let us prove that $PQ$ is a
continuous linear functional on the $k$-fold symmetric tensor
product of $E$ with the injective symmetric norm
($\varepsilon^s_k$). Take $\psi=\sum_i
x_i^{k+l}\in\TS{k+l}_{\varepsilon_{k+l}^s}E$, then
\begin{eqnarray*}
\big|\langle PQ,\psi\rangle\big| &=& \big|\sum_i P(x_i)Q(x_i)\big| = \big|P\big(\sum_i x_i^kQ(x_i)\big)\big| \le \|P\|_I\sup_{\gamma\in B_{E'}}\big|\sum_i \gamma(x_i)^kQ(x_i)\big| \\
&=& \|P\|_I\sup_{\gamma\in B_{E'}}\Big|Q\Big(\sum_i
\gamma(x_i)^kx_i^l\Big)\Big|  \le  \|P\|_I\|Q\|_I\sup_{\gamma\in
B_{E'}}\sup_{\varphi\in B_{E'}}\big|\sum_i
\gamma(x_i)^k\varphi(x_i)^l\big|.
\end{eqnarray*}
Let $R\in\p(^{k+l}E')$, $R(\gamma):=\sum \gamma(x_i)^{k+l}$ for
$\gamma\in E'$. Then by \cite[Corollary 4]{Har97}, if we take
$\gamma,\varphi\in S_{E'}$, we obtain
\begin{eqnarray*}
|\v R(\underbrace{\gamma,\dots,\gamma}_{k},\underbrace{\varphi,\dots,\varphi}_l)| &=&\big|\sum_i \gamma(x_i)^k\varphi(x_i)^l\big| \le \frac{(k+l)^{k+l}}{(k+l)!}\frac{k!}{k^k}\frac{l!}{l^l} \|R\|\\
&=&\frac{(k+l)^{k+l}}{(k+l)!}\frac{k!}{k^k}\frac{l!}{l^l}
\varepsilon_{k+l}^s\Big(\sum_i x_i^{k+l}\Big).
\end{eqnarray*}
Therefore,
$$
\big|\langle PQ,\psi\rangle\big| \leq
\frac{(k+l)^{k+l}}{(k+l)!}\frac{k!}{k^k}\frac{l!}{l^l}\|P\|_I\|Q\|_I\varepsilon_{k+l}^s(\psi),
$$
and so $PQ$ is integral with $\|PQ\|_I\le
\frac{(k+l)^{k+l}}{(k+l)!}\frac{k!}{k^k}\frac{l!}{l^l}\|P\|_I\|Q\|_I\leq
e^{k+l}\|P\|_I\|Q\|_I$.

S. Lassalle and C. Boyd proved that the product of Banach algebra
valued integral polynomials is integral (personal communication).
\end{enumerate}
\end{example}

\bigskip

One may wonder if any coherent sequence is automatically multiplicative. The answer is no. The construction in  \cite[Section 2]{CarDimMur09} can be easily adapted to obtain  a coherent sequence  $\{\u_n\}_n$ with $\u_1=\mathcal{L}$, $\u_2=\mathcal{P}^2$, $\u_3=\mathcal{P}^3$ and  $\u_4=\mathcal{P}^4_{wsc0}$ (the
$4$-homogeneous polynomials that are weakly sequentially
continuous at 0). To see that the
sequence is not multiplicative consider, for instance,
$P\in\mathcal{P}^2(\ell_2)$ given by $P(x)=\sum_nx_n^2$. Then
$P\in \u_2(\ell_2)$ but $P^2\not\in\u_4(\ell_2)$.

\bigskip

Suppose we have a sequence of ideals which is related to a sequence
of tensor norms. In this case, the multiplication property of the
ideal has a translation into properties of the tensor norms, as
shown in Proposition~\ref{maximini} below. First we need the
following proposition, that we believe is of independent interest.
Recall that a normed ideal of linear operators is said to be closed
if the norm considered is the usual operator norm.

\begin{proposition}\label{composition}
Let $\s$ be a sequence of polynomial ideals and $\mathfrak C$ be a
closed ideal of operators. Suppose that there exists a constant $M$ such that for each
$P\in\u_k(E)$ and $Q\in\u_l(E)$ it holds that $PQ\in\u_{k+l}(E)$ with
$$
\|PQ\|_{\u_{k+l}(E)}\le M\|P\|_{\u_{k}(E)}\|Q\|_{\u_{l}(E)}.
$$
Then, the sequence $\{\u_k\circ \mathfrak C\}_k$ has the same
property.
\end{proposition}
\begin{proof}Take $P\in\u_k\circ \mathfrak C(E)$ and $Q\in\u_l\circ \mathfrak C(E)$
and write them as $P=\tilde P\circ S$ and $Q=\tilde Q\circ T$, with
$S\in \mathfrak C(E,E_1)$, $T\in \mathfrak C(E,E_2)$,
$\|S\|_{\mathfrak C(E,E_1)}=\|T\|_{\mathfrak C(E,E_2)}=1$, $\tilde
P\in \u_k(E_1)$ and $\tilde Q\in \u_l(E_2)$. We  consider the
product space $E_1\times E_2$ with the supremum norm and define
$\tilde S:E\to E_1\times E_2$ and $\tilde T:E\to E_1\times E_2$ by
$\tilde S(x)=(S(x),0)$ and $\tilde T(x)=(0,T(x))$. Clearly, $\tilde
S$ and $\tilde T$ belong to $\mathfrak C(E,E_1\times E_2)$ and so is
$\tilde S+\tilde T$. Moreover, the norm of $\tilde S+\tilde T$ in
$\mathfrak C$ is the maximum of those of $S$ and $T$, thus $\|\tilde
S+\tilde T\|_{\mathfrak C(E,E_1\times E_2)}=1$.

On the other hand, in a similar way we can see that $R:E_1\times
E_2\to \mathbb K $ given by $R(y_1,y_2)=\tilde P(y_1)\tilde
Q(y_2)$ belongs to $\u_{k+l}(E_1\times E_2)$, and
$\|R\|_{\u_{k+l}(E_1\times E_2)}\le M\|\tilde
P\|_{\u_k(E_1)}\|\tilde Q\|_{\u_l(E_2)}$. Since $PQ=R\circ(\tilde
S+\tilde T)$, we have that $PQ$ belongs to $\u_{k+l}\circ
\mathfrak C$. Moreover,
$$\|PQ\|_{\u_{k+l}\circ \mathfrak C(E)}\le
\|R\|_{\u_{k+l}(E_1\times E_2)}\|\tilde S+\tilde T\|_{\mathfrak
C(E,E_1\times E_2)}\le M\|\tilde P\|_{\u_k(E_1)}\|\tilde
Q\|_{\u_l(E_2)}.$$
Considering all the possible factorizations of
$P$ and $Q$ (with operators of norm 1) we obtain the desired norm
estimate.
\end{proof}

\begin{corollary}
Let $\s$ be a multiplicative sequence and $\mathfrak C$ a closed
ideal of operators. Then $\{\u_k\circ \mathfrak C\}_k$ is a
multiplicative sequence.
\end{corollary}

\begin{proof} Just combine Proposition~\ref{composition} and \cite[Proposition 3.1]{CarDimMur09}.
\end{proof}

Now we turn our attention to tensor norms. We say that a polynomial ideal $\u_k$ is associated to the finitely generated $k$-fold  symmetric tensor norm $\alpha_k$ if for each finite dimensional normed space $M$ we have the isometry
$$  \u_k(M) :\overset 1 = \big( \otimes^{k,s} M', \alpha_k \big).$$ It is clear that each tensor norm has unique maximal and minimal associated ideals.

The following result is the announced translation of the multiplication property into a tensorial setting. By $\sigma$ we denote the symmetrization operator on the corresponding tensor product.

\begin{proposition}\label{maximini}
For each $k$ natural numbers, let $\alpha_k$ be finitely generated $k$-fold symmetric tensor norm $\alpha_k$. Consider $\u^{max}_k$ and $\u^{min}_k$ the maximal and minimal ideals associated to $\alpha_k$. Fixed $c_{k,l}>0$, the following assertions
are equivalent.
\newline ($i$) For every Banach space $E$, if $P\in\u^{max}_k(E)$ and $Q\in\u^{max}_l(E)$ then
$PQ\in\u^{max}_{k+l}(E)$ and
$$
\|PQ\|_{\u^{max}_{k+l}(E)}\le c_{k,l}\|P\|_{\u^{max}_{k}(E)}\|Q\|_{\u^{max}_{l}(E)}.
$$
\newline ($ii$) For every Banach space $E$, if $P\in\u^{min}_k(E)$ and $Q\in\u^{min}_l(E)$ then
$PQ\in\u^{min}_{k+l}(E)$ and
$$
\|PQ\|_{\u^{min}_{k+l}(E)}\le c_{k,l}\|P\|_{\u^{min}_{k}(E)}\|Q\|_{\u^{min}_{l}(E)}.
$$
($iii$) For every Banach space $E$, if $s\in\TS{k}_{\alpha_k}E'$
and $t\in\TS{l}_{\alpha_l}E'$, then
$$\alpha_{k+l}(\sigma(s\otimes t))\le c_{k,l}\alpha_k(s)\alpha_l(t).$$

\end{proposition}

\begin{proof}
The three statements are clearly equivalent if $E$ is a finite dimensional Banach space. By the very definition of maximal polynomial ideals \cite{FloHun02}, if ($i$) holds for finite dimensional Banach spaces, then it also holds for every Banach space. As a consequence, $(i)$ is implied by either ($ii$) or by ($iii$).

We now prove that ($i$) implies ($iii$).
Note that ($iii$) is equivalent to prove that the bilinear map $\phi_E:\Big(\TS{k}_{\alpha_k}E'\times\TS{l}_{\alpha_l}E',\|\cdot\|_\infty\Big)\to  \TS{k+l}_{\alpha_{k+l}}E$, $\phi_E(s,t)=\sigma(s\otimes t)$ is continuous of norm $\le c_{k,l}$ for every Banach space $E$. If ($i$) is true then $\phi_S$ is continuous (with norm $\le c_{k,l}$) for every finite dimensional Banach space $S$.
Let $M,N$ be two finite dimensional subspaces of $E'$ such that $s\in\TS{k}M$ and $t\in\TS{l}N$. Then
\begin{eqnarray*}
\alpha_{k+l}\big(\sigma(s\otimes t),\TS{k+l}M+N\big) &\le & c_{k,l}\max\{\alpha_{k}\big(s,\TS{k}M+N\big),\alpha_{l}\big(t,\TS{l}M+N\big)\} \\
& \le & c_{k,l}\max\{\alpha_{k}\big(s,\TS{k}M\big),\alpha_{l}\big(t,\TS{l}N\big)\},
\end{eqnarray*}
where the second inequality is true by the metric mapping property.
Taking the infimum over $M$ and $N$ we obtain that $\|\phi_E\|\le
c_{k,l}$ and thus we have ($iii$).

To see that $(i)$ implies $(ii)$, just note that $\u^{min}_k=\u^{max}_k\circ\overline{\mathcal F}$ and use Proposition~\ref{composition}.
\end{proof}

\begin{remark}\label{analogously}
(a) Note that in the proof of the previous proposition, we have shown that if $\s$ or $\{\alpha_k\}_k$ satisfy any of the three conditions on spaces of finite dimension, then the three statements of previous proposition hold for every Banach space.

(b) The analogous statement holds for the coherence conditions in
Definition~\ref{deficoherente}.

(c) Condition (iii) is one of the inequalities fulfilled by a
``family of complemented symmetric seminorms'', defined by C. Boyd
and S. Lassalle in \cite{BoyLas08}.

\end{remark}

\begin{corollary}\label{min y max}
Let $\s$ be a multiplicative sequence. Then $\{\u_k^{min}\}_k$ and
$\{\u_k^{max}\}_k$ are multiplicative sequences.
\end{corollary}

In \cite{CarGal-4}, natural tensor norms for arbitrary order are
introduced and studied, in the spirit of the natural tensor norms of
Grothendieck. Let us see that the polynomial ideals associated to
the natural symmetric tensor norms are multiplicative. First we
introduce some notation.

For a symmetric tensor norm $\beta_k$ (of order $k$), the projective
and injective associates (or hulls) of $\beta_k$ will be  denoted,
by extrapolation of the 2-fold case, as $\bs \beta_k /$ and $/
\beta_k \bs$ respectively. They are defined as the tensor norms
induced by the following mappings (see \cite[p. 489]{DefFlo93}):
$$ \big( \otimes^{k,s} \ell_1(B_E),  \beta_k  \big) \overset 1 \twoheadrightarrow \big( \otimes^{k,s}  E,   \bs \beta_k /  \big).$$
$$ \big( \otimes^{k,s} E, / \beta_k \bs \big) \overset 1 \hookrightarrow  \big( \otimes^{k,s} \ell_{\infty}(B_{{E}'}),  \beta_k \big).$$

Recall that for a symmetric tensor norm $\beta_k$, its dual tensor norm $\beta_k'$ is defined on finite dimensional normed spaces by
$$  \big( \otimes^{k,s} M, \beta_k' \big) :\overset 1 = [\big( \otimes^{k,s} M', \beta_k \big)]'$$
and then extended to Banach spaces so that it is finitely generated (see \cite[4.1]{Flo01}).

We say that  $\beta_k$ is a natural symmetric tensor norm (of order
$k$) if $\beta_k$ is obtained from $\pi_k$ with a finite number of
the operations $\setminus  \ /$, $/ \ \setminus$, $'$.

 For $k\ge 3$, it is
shown in \cite{CarGal-4} that there are exactly six non-equivalent
natural tensor norms (note that for $k=2$ there are only four). They
can be arranged in the following diagram:
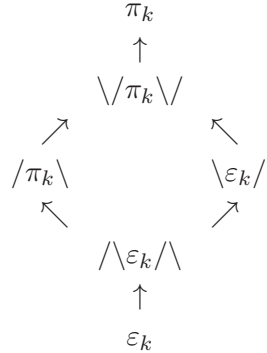
\begin{equation*}
\begin{array}{rcl}
 & \pi_k &  \\
 &  \uparrow &  \\
 & \bs / \pi_k \bs / &  \\
 \nearrow & &\nwarrow  \\
/ \pi_k \bs &  & \bs \varepsilon_k / \\
\nwarrow & &  \nearrow \\
& / \bs \varepsilon_k / \bs & \\
 &  \uparrow &  \\
& \varepsilon_k & \\
\end{array}
\end{equation*}
where $\alpha_k \to \gamma_k$ means that $\gamma_k$ dominates
$\alpha_k$. There are no other dominations.

Therefore, we have six ``natural
sequences''  $\{\alpha_k\}_k$ of symmetric tensor norms, with their corresponding associated polynomial ideals.
If, as usual, we denote $/\pi_k \bs$ by $\eta_k$, we have
$\eta_k'=\bs \varepsilon _k/$, $\bs \eta_k / =\bs / \pi_k \bs / $
and $/ \eta_k' \bs =/ \bs \varepsilon_k / \bs$.

For the  multiplicativity of the polynomial ideals associated to the
natural symmetric tensor norms we need the following:

\begin{lemma}For each $k$, let $\alpha_k$ be a finitely generated $k$-fold symmetric tensor and suppose there is a constant $c_{k,l}>0$ such that
$\alpha_{k+l}(\sigma(s\otimes t))\le c_{k,l}\alpha_k(s)\alpha_l(t)$
for every $s\in\TS{k}_{\alpha_k}E'$ and $t\in\TS{l}_{\alpha_l}E'$.
Then the same inequality holds for the sequences $\{/ \alpha_k \bs
\}_k$ and $\{\bs \alpha_k /  \}_k$.
\end{lemma}
\begin{proof}
The argument used in \cite[p.20]{BoyLas08} to show that $\{\eta_k\}_k$
is a complemented sequence of seminorms can be readily followed to
show the statement for $\{/ \alpha_k \bs\}_k$. Indeed, if, for all
$k$, we denote
$$
i_k=\otimes^{k}i : \big( \otimes^{k,s} E', / \alpha_k \bs \big) \overset 1
\hookrightarrow \big( \otimes^{k,s} \ell_{\infty}(B_{{E}''}),
\alpha_k \big),
$$
then
\begin{eqnarray*}
/ \alpha_{k+l} \bs(\sigma(s\otimes t)) &=&
\alpha_{k+l}\Big(i_{k+l}\big( \sigma(s\otimes t)\big)\Big)
=\alpha_{k+l}\Big(\sigma\big( i_k(s)\otimes i_l(t)\big)\Big)\\
&\leq &c_{k,l} \alpha_k( i_k(s))\alpha_l( i_l(t))= c_{k,l}/ \alpha_k
\bs (s) / \alpha_l \bs (t).
\end{eqnarray*}

 On the other hand, if $\u^{max}_k$ is the  maximal  ideal
associated to $\alpha_k$, then by Proposition \ref{maximini}, if $P\in\u^{max}_k(E)$ and $Q\in\u^{max}_l(E)$ then
$PQ\in\u^{max}_{k+l}(E)$ and $\|PQ\|_{\u^{max}_{k+l}(E)}\le c_{k,l}\|P\|_{\u^{max}_{k}(E)}\|Q\|_{\u^{max}_{l}(E)}$.
Moreover, the identity  $\bs \alpha_k / = (/
\alpha'_k \bs)'$ and the representation theorem for maximal
polynomial ideals \cite[Section 3.2]{FloHun02} show
that the maximal polynomial ideal $\mathfrak B_k$ associated to $\bs
\alpha_k /$ at $E$ is $\Big(\TS{k}_{/
\alpha'_k \bs}E\Big)'$. Thus, since $\TS{k}_{/
\alpha'_k \bs}E$ is (isometrically) a subspace of $\TS{k}_{
\alpha'_k }\ell_{\infty}(B_{E'})$, by the Hahn-Banach theorem $\mathfrak B_k(E)$ consists of all $k$-homogeneous polynomials on $E$ which extend to $\alpha_k'$-continuous polynomials on $\ell_{\infty}(B_{E'})$. That is,
$$\mathfrak B_k(E)=\{P\in \p^k(E): P \text{ extends to a polynomial
}\tilde P\in \u^{max}_k(\ell_{\infty}(B_{E'}))\},$$ and the norm of
$P$ in $\mathfrak B_k$ is given by the infimum of the
$\u^{max}_k$-norms of these extensions. Then, it is easy to see
that the product of two polynomials in $\mathfrak B$ belongs to
$\mathfrak B$ with the same inequality of norms. Using Proposition
\ref{maximini} again, we obtain the desired result for $\{\bs
\alpha_k /\}_k$.
\end{proof}

From Remark \ref{analogously}(b), a statement as in the previous lemma is true for the coherence conditions.
As a consequence, since $\pi_k$ and $\varepsilon_k$
are multiplicative, we can use the previous lemma and Proposition
\ref{maximini} to show that:

\begin{theorem}
Let $\{\alpha_k\}_k$ be any of the natural sequences of symmetric
tensor norms. Then the sequences  $\{\u^{max}_k\}_k$ and
$\{\u^{min}_k\}_k$ of maximal and minimal  ideals associated to $\{\alpha_k\}_k$ are
 multiplicative.
\end{theorem}

Also, it is proved in \cite{CarDimMur07} that the interpolation of
multiplicative sequences is multiplicative.

\section{Analytic structure on the spectrum}

In \cite{AroGalGarMae96} an analytic structure in the spectrum of $H_b(U)$ ($U$ an open subset of symmetrically regular Banach space) was given and it was shown that the functions in $H_b(U)$ have analytic extension to the spectrum. For the case of entire functions Dineen proved in \cite[Section 6.3]{Din99} that the extensions to the spectrum are actually of bounded type in each connected component of the spectrum.

In this section we will show that it is possible to attach an
analogous analytic structure to the spectrum $M_{b\u}(E)$ of
$H_{b\u}(E)$ for a wide class of Banach spaces $E$ and
multiplicative sequences $\u$. Then the spectrum turns out to a
Riemann domain spread over $E''$ and, as in \cite{AroGalGarMae96} or
\cite{Din99}, each connected component of $M_{b\u}(E)$ is
an analytic copy of $E''$. So we are able to define functions of the
class $H_{b\u}$ on each connected component of $M_{b\u}(E)$. It follows that with an additional
condition which is fulfilled for most of our examples we can prove
that functions in $H_{b\u}(E)$ extend to $\u$-holomorphic functions of bounded
type on each connected component of $M_{b\u}(E)$.

For $\u$ a multiplicative sequence, let us consider the spectrum $M_{b\u}(E)$ of the algebra $H_{b\u}(E)$ (i.e. the
set of continuous nonzero multiplicative functionals on $H_{b\u}$).
Since the inclusion $H_{b\u}(E)\hookrightarrow H_{b}(E)$ is continuous,
evaluations at points of $E''$ belong to $M_{b\u}(E)$.
Therefore,
$\delta_z$ is a continuous homomorphism for each $z\in E''$ and we can see $E''$ as a subset of $M_{b\u}(E)$.

Also, given $\varphi\in M_{b\u}(E)$ we can define an  element
$\pi(\varphi)\in E''$ by $\pi(\varphi)(\gamma)=\varphi(\gamma)$ for
every $\gamma\in E'$. Then the linear mapping
\begin{eqnarray*}
\pi:M_{b\u}(E)&\to& E''\\
\varphi &\mapsto & \varphi|_{E'}
\end{eqnarray*}
is a projection from $M_{b\u}(E)$ onto $E''\subset M_{b\u}(E)$. From the definition of $\pi$, for
$\varphi\in M_{b\u}(E)$ and $\gamma\in E'$ we have
$$
\varphi(\gamma^N)=\varphi(\gamma)^N=\Big(\pi(\varphi)(\gamma)\Big)^N
= \Big(AB(\gamma)(\pi(\varphi))\Big)^N = AB(\gamma^N)(\pi(\varphi)).
$$
Thus, for every finite type polynomial $P$,
$$
\varphi(P)=AB(P)(\pi(\varphi))=\delta_{\pi(\varphi)}(P).
$$
As a consequence, we have the following:
\begin{lemma}
Let $\u$ be a multiplicative sequence and $E$ a Banach space such that, for
every $k$, the finite type $k$-homogeneous polynomials are dense in
$\u_k(E)$. Then $M_{b\u}(E)=E''$.
\end{lemma}

\begin{example} \rm
Since the finite type polynomials are dense in any minimal ideal, if
$\u$ is a multiplicative sequence of minimal ideals, then
$M_{b\u}(E)=E''$ for any Banach space $E$. In particular, this happens for the nuclear
 and the approximable polynomials, so $M_{bN}(E)=E''$ and
$M_{ba}(E)=E''$.
\end{example}

The Aron-Berner extension plays a crucial role in the analytic
structure of $M_b(E)$ given in \cite{AroGalGarMae96}. In order to
obtain a similar structure for our algebras, we need the polynomial
ideals to have a good behavior with these extensions. So let us
introduce the following:

\begin{definition}\label{AB coherente}\rm
A sequence $\u$ of scalar valued ideals of polynomials is said to be
 \textbf{$AB$-closed} if there exists a constant
$\alpha>0$ such that for each Banach space $E$, $k\in\mathbb N$ and
$P\in\u_k(E)$ we have that $AB(P)$ belongs to $\u_k(E'')$ and
$\|AB(P)\|_{\u_k(E'')}\le\alpha^k\|P\|_{\u_k(E)}$, where $AB$
denotes the Aron-Berner extension.
\end{definition}

\begin{example}\rm
The sequence $\u$ is known to be $AB$-closed with constant
$\alpha=1$ in the following cases: continuous polynomials
$\u=\{\p^k\}_k$ (see \cite{DavGam89}), integral polynomials
$\u=\{\p^k_I\}_k$ (see \cite{CarZal99}), extendible polynomials
$\u=\{\p^k_e\}_k$  (see \cite{Car99}), weakly continuous on bounded
sets polynomials $\u=\{\p^k_w\}_k$  (see \cite{Mor84}), nuclear
polynomials $\u=\{\p^k_N\}_k$, approximable polynomials
$\u=\{\p^k_a\}_k$.
\end{example}

In \cite{CarGal-3} it is shown that if $\u_k$ is a maximal or a minimal ideal, then the Aron-Berner extension is an isometry from $\u_k(E)$ into $\u_k(E'')$, extending a well known result of Davie and Gamelin~\cite{DavGam89} and analogous results for some particular polynomial ideals. Therefore, any sequence $\s$ of  maximal (or minimal) polynomial ideals is $AB$-closed with constant $\alpha=1$.
Note that all the previous examples but $\u=\{\p^k_w\}_k$ are maximal or minimal, so they are covered by this result.

\begin{example}\rm
It is easy to prove that if the sequences $\{\u_k^0\}_k$ and
$\{\u_k^1\}_k$ are $AB$-closed with constants $\alpha_0$ and
$\alpha_1$ respectively, then the interpolated sequence
$\{\u_k^{\theta}\}_k$ is $AB$-closed with constant
$\alpha_0^{1-\theta}\alpha_1^{\theta}$.
\end{example}

\begin{remark}\label{remark AB coherente}\rm
Note that as a consequence of the above definition, if $\u$ is coherent and $AB$-closed, for each
$P\in\u_k(E)$, $j<k$ and $z\in E''$, we have that
$AB(P)_{z^{k-j}}\in\u_j(E'')$ and $\|AB(P)_{z^{k-j}}\|_{\u_j(E'')}
\le (C\|z\|)^{k-j}\alpha^k\|P\|_{\u_k(E)}$.

Moreover, since the $\u_k$'s are ideals, if $Q\in\u_j(E'')$ then
$Q\circ J_E\in\u_j(E)$ and $\|Q\circ
J_E\|_{\u_j(E)}\le\|Q\|_{\u_j(E'')}$. Therefore for each
$P\in\u_k(E)$, $AB(P)_{z^{k-j}}\circ J_E\in\u_j(E)$ and
$\|AB(P)_{z^{k-j}}\circ J_E\|_{\u_j(E)}\le
(C\|z\|)^{k-j}\alpha^k\|P\|_{\u_k(E)}$, where $J_E$ denotes the canonical injenction of $E$ into $E''$.

Also note that if $f\in H_{b\u}(E)$ then $AB(f)\in H_{b\u}(E'')$ and
$p_R(AB(f))\le p_{\alpha R}(f)$.
\end{remark}


 We want now to define a topology on
$M_{b\u}(E)$ which makes $(M_{b\u}(E),\pi)$ into a Riemann domain. To do this, we need first to prove that the translation is well define in spaces of holomorphic functions associated to polynomial ideals.

\begin{lemma}\label{tau_x}
Let $\u$ be a multiplicative sequence, $E$ a Banach space and $x\in
E$. Then
$$
\begin{array}{cccc}
\tau_x : & H_{b\u}(E) & \to & H_{b\u}(E) \\
         &    f       & \mapsto & f(x+\cdot)
\end{array}
$$
is a continuous operator. In particular, if $\varphi\in H_{b\u}(E)'$ then $\varphi\circ\tau_x\in H_{b\u}(E)'$
and if $\varphi\in M_{b\u}(E)$ then $\varphi\circ\tau_x\in M_{b\u}(E)$.
\end{lemma}

\begin{proof}
Take $f=\sum_{k=0}^{\infty}P_k\in H_{b\u}(E)$ and $x\in E$. Then $P_k(x+y)=\sum_{j=0}^k\binom{k}{j}(P_k)_{x^{k-j}}(y)$ and
 $\displaystyle \tau_xf=\sum_{k=0}^\infty\sum_{j=0}^k\binom{k}{j}(P_k)_{x^{k-j}}$. Using that the sequence is coherent it is easy to see that this series converges absolutely:
\begin{equation}
 \sum_{k=0}^\infty\sum_{j=0}^k\binom{k}{j}\big\|(P_k)_{x^{k-j}}\big\|_{\u_{j}(E)}\le \sum_{k=0}^\infty (1+C\|x\|)^k\|P_k\|_{\u_k(E)}=p_{_{1+C\|x\|}}(f).
\end{equation}
 So we can reverse the order of summation to obtain that $\frac{d^j\tau_xf(0)}{j!}=\sum_{k=j}^{\infty}\binom{k}{j}(P_k)_{x^{k-j}}$. Thus we obtain that
\begin{eqnarray*}
p_R(\tau_xf) & = &\sum_{j=0}^{\infty} R^j\Big\|\frac{d^j\tau_xf(0)}{j!}\Big\|_{\u_j(E)} \le \sum_{j=0}^{\infty} R^j\sum_{k=j}^{\infty}\binom{k}{j}\big\|(P_k)_{x^{k-j}}\big\|_{\u_j(E)} \\
 & = & \sum_{k=0}^{\infty} \|P_k\|_{\u_k(E)}\sum_{j=0}^{k}\binom{k}{j}R^j(C\|x\|)^{k-j}
\le p_{_{R+C\|x\|}}(f).
\end{eqnarray*}
Therefore $\tau_xf\in H_{b\u}(E)$ and $\tau_x$ is continuous.
\end{proof}

If $\u$ is $AB$-closed and coherent and $z\in E''$ we can define
$$
\begin{array}{cccc}
\tilde\tau_z : & H_{b\u}(E) & \to & H_{b\u}(E) \\
         &    f       & \mapsto & \tau_z(AB(f))\circ J_E.
\end{array}
$$
Remark \ref{remark AB coherente} ensures that  $\tilde\tau_z$ is a
(well-defined) continuous operator and $p_R(\tilde\tau_zf)\le
p_{\alpha(R+C\|z\|)}(f)$.

\begin{corollary}\label{phi tau z}
Let $\u$ be an $AB$-closed and multiplicative sequence and let $z\in
E''$. Then $\tilde\tau_z$ is a continuous operator. Consequently, if
$\varphi\in H_{b\u}(E)'$ then $\varphi\circ\tilde\tau_z\in
H_{b\u}(E)'$ and if $\varphi\in M_{b\u}(E)$ then
$\varphi\circ\tilde\tau_z\in M_{b\u}(E)$.
\end{corollary}

Note that $\pi(\varphi\circ\tilde\tau_z)(\gamma)=\varphi\circ\tilde\tau_z(\gamma)=\varphi(AB(\gamma)(z+J_E(\cdot))= \varphi(1)z(\gamma)+\varphi(\gamma)=z(\gamma)+\varphi(\gamma)$, and thus $\pi(\varphi\circ\tilde\tau_z)=z+\pi(\varphi)$.

A necessary condition to obtain the analytic structure of the spectrum of  $H_b(E)$ is that the space $E$ be symmetrically regular (i.e. the
Arens extensions of every symmetric multilinear form are
symmetric). In our case, to study the spectrum of $H_{b\u}(E)$, we
need that the Arens extensions of $\v P$ be symmetric for every $P$ in $\u_k(E)$ (and
for all $k$). This happens, of course, if $E$ is symmetrically regular, but also for arbitrary $E$ if $\u_k$ are good enough. So we define:

\begin{definition}
A sequence $\u$ is regular at $E$ if, for every
$k$ and every $P$ in $\u_k(E)$, we have that every Arens extension (that is, any  extension by $w^*$-continuity in each variable in some order) of  $\v P$ are
symmetric.
 We say that the sequence $\u$ is regular if it is regular at $E$ for every Banach space $E$.
\end{definition}

\begin{example}\label{AB simetrica}\rm
\begin{enumerate}
\item[(a)] Any sequence of ideals contained in the ideals of approximable polynomials is regular. In particular, any sequence of minimal ideals is regular.

\item[(b)] The sequences of integral \cite[Proposition 2.14]{CarLas04},
extendible \cite[Proposition 2.15]{CarLas04} and weakly continuous
\cite{AroHerVal83} multilinear forms are regular.

\item[(c)] If $\{\alpha_k\}_k$ is a sequence of projective symmetric tensor norms and $\u$ is a sequence of ideals associated to
$\{\alpha_k\}_k$, then $\u$ is regular. Indeed, $\bs \varepsilon_k/\le \alpha_k$ and thus $\alpha_k'\le\bs \varepsilon_k/'=\eta_k$. Then,
if we denote
$\beta_k=\alpha_k'$, since $\beta_k \leq \eta_k$, every $P\in
\left(\TS{k}_{\beta_k}E\right)'$ is extendible and, by (b), any Arens extension of $\v P$ is symmetric. This says that the sequence of maximal
ideals associated to $\{\alpha_k\}_k$ is regular and so is $\u$.

\item[(d)] Let $\{\alpha_k\}_k$ be a sequence of symmetric tensor norms and let $\u$ be the sequence
of maximal polynomial ideals associated to $\{\alpha_k\}_k$. If $\u$
is regular then is regular also any sequence of polynomial ideals
associated to $\{/ \alpha_k \bs \}_k$. Indeed, if we denote
$\beta_k=\alpha_k'$, for any $P\in
\left(\otimes_{\bs\beta_k/}^{k,s}E\right)'$ we have that
$Q=P\circ \otimes^kq\in
\left(\otimes_{\beta_k}^{k,s}\ell_1(B_E)\right)'=\u_k(\ell_1(B_E))$,
where $q$ is the metric projection
$\ell_1(B_E) \twoheadrightarrow E$.
We claim that any Arens extension of $\v P$ is symmetric. Indeed, let $\sigma$ be a permutation of $\{1,\dots,k\}$ and let $\widetilde{ P}$, $\widetilde{ Q}$ the Arens extensions of $\v P$, $\v Q$ defined by $w^*$-continuity in the order determined by $\sigma$.  For $j=1,\dots,k$, take $z_j\in E''$ and bounded nets $(x_{\lambda_j}^j)_{\lambda_j}\subset E$ such that $x_{\lambda_j}^j\overset{w^*}{\to}z_j$. Then
\begin{eqnarray*} \widetilde{ Q}(z_1,\dots,z_k) & = & \lim_{\lambda_{\sigma(1)}}\dots\lim_{\lambda_{\sigma(k)}}\v Q(x_{\lambda_1}^{1},\dots,x_{\lambda_k}^{k}) =\lim_{\lambda_{\sigma(1)}}\dots\lim_{\lambda_{\sigma(k)}} \v P(q(x_{\lambda_1}^{1}),\dots,(x_{\lambda_k}^{k}))\\ & = & \widetilde{ P}(q''(z_1),\dots,q''(z_k)).\end{eqnarray*} Since $q''$ is surjective and $\widetilde{ Q}$ is symmetric, we conclude that $\widetilde{ P}$ is symmetric.

\item[(e)] As a consequence of (c) and (d) we obtain that if $\u$ is a sequence of polynomial ideals
associated with any of the natural sequences (except for the case
$\alpha_k=\varepsilon_k$) then $\u$ is regular.

\item[(f)] If the sequences $\{\u_k^0\}_k$ and
$\{\u_k^1\}_k$ are regular then the interpolated sequence
$\{\u_k^{\theta}\}_k$ is also regular (because each $\u_k^{\theta}$
is contained in $\u_k^0+\u_k^1$).
\end{enumerate}
\end{example}

As in \cite{AroGalGarMae96} it can be proved that, if $\u$ is regular at $E$ then every evaluation at a point of $E^{iv}$ is in fact an evaluation at a point of $E''$.

The next two lemmas can be obtain just as in  \cite[Pages 428-430]{Din99}.
\begin{lemma}\label{tau z tau w=tau z+w}
Let $\u$ be an $AB$-closed coherent sequence which is regular at a
Banach space $E$.  Then
$\tilde\tau_z\circ\tilde\tau_w=\tilde\tau_{z+w}$ for every $z,w\in
E''$.
\end{lemma}

\begin{lemma}
Let $\u$ be an $AB$-closed multiplicative sequence which is regular
at a Banach space $E$. For each  $\varphi\in M_{b\u}(E)$ and
$\varepsilon>0$ define
$V_{\varphi,\varepsilon}=\{\varphi\circ\tilde\tau_z:\, z\in E'',\,
\|z\|<\varepsilon\}$. Then
 $\big\{V_{\varphi,\varepsilon}:\, \varphi\in M_{b\u}(E),\,\varepsilon>0\big\}$ is the basis of a Hausdorff topology in $M_{b\u}(E)$.
\end{lemma}

\begin{proposition}\label{dominiodeRiemann}
Let $\u$ be an $AB$-closed multiplicative sequence which is regular
at a Banach space $E$. Then $(M_{b\u}(E),\pi)$ is a Riemann domain
over $E''$ and each connected component of $(M_{b\u}(E) ,\pi)$ is
homeomorphic to $E''$.
\end{proposition}
\begin{proof}
With the topology defined in the above lemma, it is clear that for each $\varphi\in M_{b\u}(E)$ and $\varepsilon>0$, $\pi|_{V_{\varphi,\varepsilon}}$ is an homeomorphism onto $B_{E''}(\pi(\varphi),\varepsilon)$. Thus $\pi:M_{b\u}(E)\to E''$ is a local homeomorphism. Note that given $\varphi\in M_{b\u}(E)$, by Corollary \ref{phi tau z}, $\varphi\circ\tilde\tau_z$ is an homomorphism for each $z\in E''$. Moreover, since $\pi(\varphi\circ\tilde\tau_z)=\pi(\varphi)+z$ it follows that $\pi$ is an homeomorphism from $S(\varphi):=\{\varphi\circ\tilde\tau_z:\, z\in E''\}$ to $E''$ and thus $S(\varphi)$ is the connected component of $\varphi$ in $M_{b\u}(E)$.
\end{proof}

\begin{example}\rm
$(M_{b\u}(E),\pi)$ is a Riemann domain over $E''$ (and each
connected component is homeomorphic to $E''$) in the following
cases:

\begin{enumerate}
\item[(a)] $\u=\{\p^k_I\}_k$ or $\u=\{\p^k_e\}_k$ or, more
generally, $\u$ the sequence of maximal polynomial ideals associated
to any of the natural sequences $\{\alpha_k\}_k$ (except for
$\alpha_k=\epsilon_k$) and $E$ any Banach space.

\item[(b)] $\u=\{\p^k_w\}_k$ and $E$ any Banach space.

\item[(c)] $\u$ any multiplicative sequence of maximal
polynomial ideals and $E$ symmetrically regular.

\item[(d)] $\u=\{\u_k^{\theta}\}_k$, with $\{\u_k^0\}_k$ and
$\{\u_k^1\}_k$ any of the sequences of the examples (a) or (b) (or
(c) and $E$ symmetrically regular).

\end{enumerate}
\end{example}

Each function $f\in H_{b\u}(E)$ can be extended via its Gelfand transform $\tilde f$ to the spectrum $M_{b\u}(E)$, that is $\tilde f(\varphi)=\varphi(f)$. Now that we have proved that $M_{b\u}(E)$ is a Riemann domain, it is natural to ask if $\tilde f$ is analytic in $M_{b\u}(E)$. Moreover, one can wonder if $\tilde f$ preserve some of the properties of $f$ in terms of the ideals $\u$.
Also, given $\varphi\in H_{b\u}(E)'$ and $f\in H_{b\u}(E)$,
Corollary~\ref{phi tau z} allows us to define a function on $E''$
by $z\mapsto \varphi\circ\tilde\tau_z(f)$. We
will show in Theorem \ref{convolucion} that this function belongs to
$H_{b\u}(E'')$. This will allow us to conclude that the restriction of $\tilde f$ to each connected component of $M_{b\u}(E)$ is of $\u$-holomorphic (Theorem~\ref{extension al espectro} below).

First, note that for $\varphi\in\hbu(E)'$, there are constants $c,r>0$ such that
$|\varphi(g)|\le cp_{_r}(g)$, for every $g\in\hbu(E)$. In
particular, $|\varphi(P)|\le cp_{_r}(P)= cr^k\|P\|_{\u_k(E)}$. As a consequence, we have
\begin{equation}\label{norma de phi}\|\varphi_{|_{\u_k(E)}}\|_{\u_k(E)'}\le cr^k,\end{equation} for every $k\ge 1$.

\medskip

When a sequence of polynomial ideals is defined in both the scalar and vector valued case, one may consider the following property: ``for every $P\in\u_k(E)$, the mapping $x\mapsto P_{x^l}$ belongs to the space of vector-valued polynomials $\u_l(E;\u_{k-l}(E))$''. This would mean that the differentials of a polynomial in $\u$ are also  polynomials in $\u$. Since we are dealing with scalar-valued polynomial ideals, we can consider a similar property, which could be read as ``the differentials of a polynomial in $\u$ are weakly in $\u$''. More precisely, we have:

\begin{definition}
Let $\u$ be a coherent sequence of polynomial ideals and let $E$ be
a Banach space. We say that $\u$ is weakly differentiable if there
exists a constant $K$ such that, for  $l < k$,  $P\in\u_k(E)$ and $\varphi\in\u_{k-l}(E)'$,  the mapping
$x\mapsto\varphi(P_{x^l})$ belongs to $\u_l(E)$ and
$$\Big\|x\mapsto\varphi\big(P_{x^l}\big)\Big\|_{\u_l(E)} \le K^k\|\varphi\|_{\u_{k-l}(E)'}\|P\|_{\u_k(E)}.$$
\end{definition}

Since $\frac{d^{k-l}P}{(k-l)!}(x)=\binom{k}{l}P_{x^l}$, the previous condition is equivalent to say that $\varphi\circ \frac{d^{k-l}P}{(k-l)!}$ belongs to $\u_l(E)$ and
$$
\Big\|\varphi\circ \frac{d^{k-l}P}{(k-l)!}\Big\|_{\u_l(E)} \le \binom{k}{l}K^k\|\varphi\|_{\u_{k-l}(E)'}\|P\|_{\u_k(E)}.
$$ This is what we mean by saying that the differentials are weakly in $\u$ and what suggested our terminology.

As the following proposition shows, there is some kind of duality between the properties of multiplicativity and weakly differentiability of a sequence.
\begin{proposition}\label{u w-diff implica u* multip}
\begin{enumerate}
\item[($i$)] Let $\u=\{\u_k\}_k$ be a weakly differentiable sequence. Then the sequence of adjoint ideals $\{\u_k^{*}\}_k$ is multiplicative.

\item[($ii$)] Let $\u=\{\u_k\}_k$ be a multiplicative sequence. Then the sequence of adjoint ideals $\{\u_k^{*}\}_k$ is weakly differentiable.
\end{enumerate}
In both cases the constants of multiplicativity and weakly differentiability are the same.
\end{proposition}
\begin{proof}
($i$) From \cite[Proposition 5.1]{CarDimMur09} we know that $\{\u_k^{*}\}_k$ is coherent. By Remark \ref{analogously}, it suffices to check that $\|PQ\|_{\u^{*}_{k+l}(M)}\le K^{k+l}\|P\|_{\u^{*}_{k}(M)}\|Q\|_{\u^{*}_{l}(M)}$ for any finite dimensional Banach space $M$. Since $M$ is finite dimensional, $\u_k^*(M)$ is just $\u_k(M')'$. Take
$P\in\u_{k}^*(M)=\u_{k}(M')'$ and $Q\in\u_{l}^*(M)=\u_{l}(M')'$. For $\Psi=\sum_j x_j^{k+l}\in\u_{k+l}(M')$, we have $$\langle PQ,\Psi\rangle =\sum_j P(x_j)Q(x_j)=\langle Q,\sum_j\langle P,x_j^k\rangle x_j^l\rangle =\langle Q,\gamma\mapsto\sum_j\langle P,x_j^k\rangle x_j(\gamma)^l\rangle =\langle Q,\gamma\mapsto P(\Psi_{\gamma^l})\rangle.$$ Thus, since $\{\u_k\}_k$ is weakly differentiable, $$|\langle PQ,\Psi\rangle| \le \|Q\|_{\u_{l}^*(M)}\ \|\gamma\mapsto P(\Psi_{\gamma^l})\|_{\u_{l}(M)}\le \|Q\|_{\u_{l}^*(M)}K^{k+l}\|P\|_{\u_{k}^*(M)}\|\Psi\|_{\u_{k+l}(M')},$$ which implies that $\|PQ\|_{\u_{k+l}^*(M)}\le K^{k+l}\|P\|_{\u_{k}^*(M)}\|Q\|_{\u_{l}^*(M)}$.

($ii$) For each $k$, let $\alpha_k$ be the finitely generated symmetric tensor norm associated to $\u_k$, so that for every $E$, $\u_k^*(E)=\Big(\bigotimes_{\alpha_k}^{k,s}E\Big)'$. By Proposition \ref{maximini} and Remark \ref{analogously} ($a$), the multiplicativity of $\u$ with constant $K$ implies that $\alpha_k(\sigma(s\otimes t))\le K^k\alpha_l(s)\alpha_{k-l}(t)$, for every $s\in\bigotimes_{\alpha_l}^{l,s}E$ and $t\in\bigotimes_{\alpha_{k-l}}^{k-l,s}E$.

Take $P\in\u_k^*(E)$ and $\varphi\in\u_{k-l}^*(E)'$, with $\|\varphi\|=1$. Define $Q(x)=\varphi(P_{x^l})$. Note that $Q$ is a well defined $l$-homogeneous polynomial since the sequence of adjoint ideals of a coherent sequence is again coherent \cite[Proposition 5.1]{CarDimMur09}. We have to prove that $Q$ belongs to $\u_l^*(E)=\Big(\bigotimes_{\alpha_l}^{l,s}E\Big)'$ and that $\|Q\|_{\u_l^*(E)}\le K^k\|P\|_{\u_k^*(E)}$. Take $s=\sum_j y_j^l\in\bigotimes_{\alpha_l}^{l,s}E$ and $\varepsilon>0$. Since $\varphi\in\Big(\bigotimes_{\alpha_{k-l}}^{k-l,s}E\Big)''$, by Goldstine theorem there is some $t=\sum_i x_i^{k-l}\in\bigotimes_{\alpha_{k-l}}^{k-l,s}E$, with $\alpha_{k-l}(t)\le1$ such that $|\langle Q,s\rangle|=|\sum_j \varphi(P_{y_j^l})|\le|\langle\sum_j P_{y_j^l},t\rangle|+\varepsilon$.
Thus, \begin{eqnarray*}|\langle Q,s\rangle| & \le & |\langle\sum_j P_{y_j^l},\sum_i x_i^{k-l}\rangle|+\varepsilon=\big|\sum_{i,j}\v P(x_i^{k-l},y_j^{l})\big|+\varepsilon=\big|\langle P,\sigma\big(\sum_{i,j}x_i^{k-l}\otimes y_j^{l})\rangle\big|+\varepsilon \\ & \le & \|P\|_{\u_k^*(E)}\alpha_k(\sigma(s\otimes t))+\varepsilon   \le  K^k\|P\|_{\u_k^*(E)}\alpha_l(s)\alpha_{k-l}(t)+\varepsilon.\end{eqnarray*} Since this is true for arbitrarily small $\varepsilon$ and $\alpha_{k-l}(t)\le1$, we conclude that
$\|Q\|_{\u_l^*(E)}\le K^k\|P\|_{\u_k^*(E)}$.
\end{proof}
Next corollary to Proposition~\ref{u w-diff implica u* multip}
follows, for maximal ideals, from the equality $\u_k^{**}=\u_k$. For
minimal ideals is a consequence of Example~\ref{ejemplos wd} (f) below.
\begin{corollary}
Let $\{\u_k\}_k$ be sequence of maximal polynomial ideals or minimal polynomial ideals. Then $\{\u_k\}_k$ is weakly differentiable (multiplicative) if and only if $\{\u_k^{*}\}_k$ is multiplicative (weakly differentiable).
\end{corollary}

If $E$ is a Banach space and $\u$ is a  weakly differentiable
coherent sequence which is $AB$-closed (with constant $\alpha$), it
easily follows that the mapping $E''\ni
z\mapsto\varphi(AB(P)_{z^l}\circ J_E)$ belongs to $\u_l(E'')$ and
\begin{equation}\label{condition}\Big\|z\mapsto\varphi(AB(P)_{z^l}\circ J_E)\Big\|_{\u_l(E'')}\le
\alpha^k K^k\|\varphi\|_{\u_{k-l}(E)'}\|P\|_{\u_k(E)}.
\end{equation}

\begin{theorem}\label{convolucion}
Let $\u$ be an $AB$-closed weakly differentiable coherent sequence.
For each $\varphi\in (H_{b\u}(E))'$, the following operator is well
defined and continuous:
$$
\begin{array}{cccl}
\tilde T_\varphi: & \hbu(E) & \to      & \hbu(E'') \\
           &    f    & \mapsto  & \left(z\, \mapsto \varphi\circ\tilde\tau_z(f)\right)
\end{array}
$$
\end{theorem}
\begin{proof}
Take $f=\sum_{k=0}^{\infty}P_k\in H_{b\u}(E)$ and $z\in E''$. Then $\varphi\circ\tilde\tau_z(f) =\sum_{k=0}^\infty\sum_{j=0}^k\binom{k}{j}\varphi\big(AB(P_k)_{z^j}\circ J_E\big)= \sum_{j=0}^\infty\sum_{k=j}^\infty\binom{k}{j}\varphi\big(AB(P_k)_{z^j}\circ J_E\big)$ since using Remark \ref{remark AB coherente} and inequality (\ref{norma de phi}) it is easy to see that this series is absolutely convergent.

Let
$Q_l(z)=\sum_{k=l}^\infty\binom{k}{l}\varphi\big(AB(P_k)_{z^l}\circ
J_E\big)$. Then $\varphi\circ\tilde\tau_z(f)=\sum_{l=0}^\infty
Q_l(z)$. We will show that $Q_l$ belongs to $\u_l(E'')$ and that
$\sum_{l=0}^\infty Q_l$ is in $H_{b\u}(E'')$. To prove this it
suffices to show that  the series
$\sum_{k=l}^\infty\binom{k}{l}\Big\|z\mapsto\varphi\big(AB(P_k)_{z^l}\circ
J_E\big)\Big\|_{\u_l(E'')}$ converges  and that for every $R>0$, the
series $\sum_{l=0}^\infty
R^l\Big\|\sum_{k=l}^\infty\binom{k}{l}z\mapsto\varphi\big(AB(P_k)_{z^l}\circ
J_E\big)\Big\|_{\u_l(E'')}$ also converges. By inequality
(\ref{condition}) we have
\begin{eqnarray*}
\sum_{l=0}^\infty R^l\Big\|\sum_{k=l}^\infty\binom{k}{l}z\mapsto\varphi\big(AB(P_k)_{z^l}\circ J_E\big)\Big\|_{\u_l(E'')} & \le & \sum_{l=0}^\infty R^l\sum_{k=l}^\infty\binom{k}{l}\Big\|z\mapsto\varphi\big(AB(P_k)_{z^l}\circ J_E\big)\Big\|_{\u_l(E'')} \\
 & \le & \sum_{l=0}^\infty R^l\sum_{k=l}^\infty\binom{k}{l} \alpha^kK^k\|\varphi_{|_{\u_{k-l}(E)}}\|_{\u_{k-l}(E)'}\|P_k\|_{\u_k(E)} \\
 & \le & c\sum_{k=0}^\infty \alpha^k\|P_k\|_{\u_k(E)}K^k\sum_{l=0}^k\binom{k}{l}R^lr^{k-l} \\
 & = & cp_{_{\alpha K(R+r)}}(f),
\end{eqnarray*}
where in the last inequality we have used (\ref{norma de phi}) and
changed the order of summation. Therefore $\tilde T_\varphi(f)$
belongs to $H_{b\u}(E'')$ and $p_{_R}(\tilde T_\varphi(f))\le
cp_{_{\alpha K(R+r)}}(f)$, that is, $\tilde T_\varphi\in\mathcal
L(H_{b\u}(E),H_{b\u}(E''))$.
\end{proof}

With a similar proof to the above result one can prove:
\begin{corollary}
Let $\u$ be a weakly differentiable multiplicative sequence. For
$\varphi\in H_{b\u}(E)'$ and $f\in H_{b\u}(E)$, we define $\varphi *
f(x)=\varphi\circ\tau_x(f)$. Then we have $\varphi * f\in
H_{b\u}(E)$ and the application
$$
\begin{array}{cccl}
T_\varphi: & H_{b\u}(E) & \to      & H_{b\u}(E) \\
           &    f    & \mapsto  & \varphi * f
\end{array}
$$
is a continuous linear operator.

Moreover, if $\psi\in M_{b\u}(E)$ we can define $\varphi*\psi\in H_{b\u}(E)'$ by $\varphi*\psi(f)=\psi\left(\varphi*f\right)$, and the application
$$
\begin{array}{cccl}
M_\varphi: & H_{b\u}(E)' & \to      & H_{b\u}(E)' \\
           &    \psi    & \mapsto  & \psi*\varphi
\end{array}
$$
is continuous.
\end{corollary}

As usual, we call a continuous operator $T:H_{b\u}(E)\to H_{b\u}(E)$ that commutes with translations  a convolution operator. We have the following characterization:

\begin{corollary}
 With the hypothesis and notation of the previous corollary, $T:H_{b\u}(E)\to H_{b\u}(E)$ is a convolution operator if and only if there exist $\varphi\in H_{b\u}(E)'$ such that $Tf=\varphi *f $ for every  $f\in H_{b\u}(E)$.
\end{corollary}
\begin{proof}
If $\varphi\in H_{b\u}(E)'$ then $T_{\varphi}(f)=\varphi*f$ is  continuous by the previous corollary and it is easily checked to be a convolution operator. Conversely, let $\varphi=T\circ\delta_0$. Then $Tf(x)=\tau_x(Tf)(0)=T(\tau_xf)(0) =\varphi * f$.
\end{proof}

Now we show that, again, our hypotheses are fulfilled by many sequences of polynomial ideals.

\begin{example}\label{ejemplos wd}\rm
The following sequences are weakly differentiable:
\begin{enumerate}
\item[(a)] $\u=\{\p^k\}_k$: if $P\in\p^k(E)$ and
$\varphi\in\p^{k-l}(E)'$ then it is clear that
$x\mapsto\varphi\big(P_{x^l}\big)\in\p^l(E)$ and
$\big\|x\mapsto\varphi\big(P_{x^l}\big)\big\|_{\p^l(E)} \le
e^l\|\varphi\|_{\p^{k-l}(E)'}\|P\|_{\p^k(E)}$.

\item[(b)] $\u=\{\p^k_I\}_k$:  this is a consequence of the above Proposition \ref{u w-diff implica u* multip} ($ii$), since $\p^k_I=(\p^k)^*$ and $\{\p^k\}$ is multiplicative with constant $M=1$.

\item[(c)] $\u=\{\p^k_e\}_k$: if $P\in\p_e^k(E)$ and $\varphi\in\p_e^{k-l}(E)'$ then $Q(x)=\varphi\big(P_{x^l}\big)$ is in $\p^l(E)$. Let $E\overset{J}{\hookrightarrow}G$ and $\tilde P$ an extension of $P$ to $G$. Then $\tilde Q(y)=\varphi\big(\tilde P_{y^l}\circ J\big)$ is an extension of $Q$ to $G$, and thus $Q$ is extendible. Moreover, since $|\tilde Q(y)|\le e^l\|y\|^l\|\varphi\|\|\tilde P\|$, it follows that $\|Q\|_e\le e^l\|\varphi\|\|P\|_e$.

\item[(d)] $\u=\{\p^k_w\}_k$: it is known (see \cite[Proposition 2.6]{Din99}) that if $P\in\p^k_w(E)$
then $d^{k-l}P$ is weakly continuous. Thus
$x\mapsto\varphi\big(P_{x^l}\big)\in\p^l_w(E)$ and has norm $\le
e^l\|\varphi\|_{\p^{k-l}(E)'}\|P\|_{\p^k(E)}$.

\item[(e)] $\u$ the sequence of maximal polynomial ideals associated
to any of the natural sequences. This follows from the multiplicativity of those sequences together with Proposition~\ref{u w-diff implica u* multip} $(ii)$, or  from the previous
examples (a) and (b) and Lemma~\ref{palito w d} below.

\item[(f)] If $\{\u_k\}_k$ is a weakly differentiable sequence and $\mathfrak C$ is a normed operator
ideal, then $\{\u_k\circ \mathfrak C\}_k$ is weakly differentiable. In particular, the minimal hulls of the ideals in a weakly differentiable sequence form also a weakly differentiable sequence.
Indeed, let $P=RT\in\u_k\circ \mathfrak C(E)$, with $T\in\mathfrak C(E,E_1)$ and $R\in\u_k(E_1)$. Take $\varphi\in\u_{k-l}\circ \mathfrak C(E)'$ and define $\psi\in\u_{k-l}(E_1)'$ by $\psi(Q)=\varphi(QT)$. Note that $\|\psi\|_{\u_{k-l}(E_1)'}\le\|\varphi\|_{\u_{k-l}\circ \mathfrak C(E)'}\|T\|^{k-l}_{\mathfrak C(E,E_1)}$.
Clearly,
$$
\varphi(P_{x^l})=\varphi(R_{(Tx)^l}T)=\psi(R_{(Tx)^l}),
$$
thus $x\mapsto\varphi(P_{x^l})=\big[x\mapsto\psi(R_{x^l})\big]\circ T$ belongs to $\u_l\circ \mathfrak C(E)$, because $\{\u_k\}$ is weakly differentiable and $T\in\mathfrak C$. Moreover
\begin{eqnarray*}
\|x\mapsto\varphi(P_{x^l})\|_{\u_l\circ \mathfrak C(E)} &=& \|\big[x\mapsto\psi(R_{x^l})\big]\circ T\|_{\u_l\circ \mathfrak C(E)} \le \|x\mapsto\psi(R_{x^l})\|_{\u_l(E_1)}\|T\|^l_{\mathfrak C(E,E_1)}\\
 &\le& K^k\|\psi\|_{\u_{k-l}(E_1)'}\|R\|_{\u_k(E_1)}\|T\|^l_{\mathfrak C(E,E_1)} \\
 &\le& K^k\|\varphi\|_{\u_{k-l}\circ \mathfrak C(E)'}\|R\|_{\u_k(E_1)}\|T\|^k_{\mathfrak C(E,E_1)}.
\end{eqnarray*}
Since this is true for every factorization of $P=RT$, we conclude that
$$
\|x\mapsto\varphi(P_{x^l})\|_{\u_l\circ \mathfrak C(E)}\le K^k\|\varphi\|_{\u_{k-l}\circ \mathfrak C(E)'}\|P\|_{\u_k\circ \mathfrak C(E)}.
$$
\end{enumerate}
\end{example}

\begin{lemma}\label{palito w d}
Let $\u$ be the sequence of maximal polynomial ideals associated to
a sequence of symmetric tensor norms $\{\alpha_k\}_k$. If $\u$ is
weakly differentiable, then the same is true for the sequences of
maximal polynomial ideals associated to $\{\bs \alpha_k/\}_k$ and to
$\{/\alpha_k\bs\}_k$.
\end{lemma}

\begin{proof}
Let us  denote $\beta_k=\alpha_k'$. For $\u$ the sequence of maximal
polynomial ideals associated to $\{\bs \alpha_k/\}_k$ we have that
$P$ belongs to $\u_k(E)$ if and only if
 $P\in
\left(\otimes_{ / \beta_k \bs }^{k,s}E\right)'$, and we can proceed
just as we did with the ideal of extendible polynomials (note that
polynomials in $\u_k(E)$ are those that extends to a
$\beta_k$-continuous polynomial on $\ell_{\infty}(B_{E'})$).

For the sequence of maximal ideals associated to
$\{/\alpha_k\bs\}_k$,  $P$ belongs to $\u_k(E)$ if and only if
$\widetilde{P}=P\circ q_k$ belongs to
$\left(\otimes_{\beta_k}^{k,s}\ell_1(B_E)\right)'=\u_k(\ell_1(B_E))$,
where $q_k$ is the metric projection
$\otimes_{\beta_k}^{k,s}\ell_1(B_E) \twoheadrightarrow
\otimes_{\bs\beta_k/}^{k,s}E$. Also, transposing $q_k$ we obtain a
metric injection $\left(\otimes_{\bs\beta_k/}^{k,s}E\right)'
\hookrightarrow \left(\otimes_{\beta_k}^{k,s}\ell_1(B_E)\right)'$.
If we take $\varphi\in \left(\otimes_{\bs\beta_k/}^{k,s}E\right)'$,
we can choose a Hahn-Banach extension $\psi$ of $\varphi$ on
$\otimes_{\beta_k}^{k,s}\ell_1(B_E)$.

Now, if $Q(x)=\varphi\big(P_{x^l}\big)$, we have $$Q\circ q(z)=
\varphi (P_{q(z)^l}) =\psi \left( (P\circ q)_{z^l}\right),$$ which
belongs to $\left(\otimes_{\beta_k}^{k-l,s}\ell_1(B_E)\right)'$
because $\u$ is weakly differentiable. But this means that $Q$
belongs to $\u_{k-l}(E)$.
\end{proof}

In the previous proof, we only used that $\u$ is weakly
differentiable in spaces of the form $\ell_\infty(I)$ (for $\{\bs
\alpha_k/\}_k$)  and $\ell_1(J)$ (for $\{/\alpha_k\bs\}_k$), where $I$
and $J$ are some index sets.

\bigskip
Now we are able to show that the extension of a function in $H_{b\u}(E)$ to the spectrum ``is $\u$-holomorphic on each connected component'':

\begin{theorem}\label{extension al espectro}
Let $E$ be a Banach space and $\u$ be an $AB$-closed multiplicative
sequence which is regular at $E$ and weakly differentiable. Then,
given any function $f\in H_{b\u}(E)$ and its extension $\tilde f$ to
$M_{b\u}(E)$, the restriction of $\tilde f$ to each connected component of $M_{b\u}(E)$ is a $\u$-holomorphic function of bounded type.
\end{theorem}
\begin{proof}
We have to show that for every $\varphi\in M_{b\u}(E)$, $\tilde
f\circ \big(\pi|_{S(\varphi)})^{-1}\in H_{b\u}(E'')$. But note that
$S(\varphi)=\{\varphi\circ\tilde\tau_z:\,z\in E''\}$ and that
$\big(\pi|_{S(\varphi)})^{-1}(z)=\varphi\circ\tilde\tau_{z-\pi(\varphi)}$
so $\tilde f\circ \big(\pi|_{S(\varphi)})^{-1}(z)
=\varphi\circ\tilde\tau_{z-\pi(\varphi)}(f)$. That is, $\tilde
f\circ
\big(\pi|_{S(\varphi)})^{-1}=\widetilde{T}_{\varphi\circ\tilde\tau_{-\pi(\varphi)}}(f)$
which is in $H_{b\u}(E'')$ by Theorem \ref{convolucion}.
\end{proof}

We can apply the last result in the following cases:
\begin{example}\rm
$(M_{b\u}(E),\pi)$ is a Riemann domain over $E''$ and every function
in $H_{b\u}(E)$ extends to an $\u$-holomorphic function of bounded
type on each connected component of $M_{b\u}(E)$  in the following
cases:
\begin{enumerate}
\item[(a)]  $\u=\{\p^k\}_k$, and $E$ is symmetrically regular (this is \cite[Proposition 6.30]{Din99}).

\item[(b)] $\u=\{\p^k_I\}_k$,  for every Banach space $E$.

\item[(c)]  $\u=\{\p^k_e\}_k$,  for every Banach space $E$.

\item[(d)] $\u=\{\p^k_w\}_k$, for every Banach space $E$.

\item[(e)] $\u$ a sequence of maximal polynomial ideals associated to
any of the natural sequences but
$\{\varepsilon_k\}_k$, for every Banach space $E$.

\end{enumerate}
\end{example}

\section{A Banach-Stone type result}

Now we apply some of our results to obtain a Banach-Stone type
theorem for algebras associated to multiplicative sequences of
polynomial ideals. We follow a procedure as in similar results in
\cite{CarGarMae05}. First, we have:

\begin{lemma}\label{la g es holo}
 Let $\u$ and $\mathfrak B$ be multiplicative sequences.
 Suppose that $\phi:H_{b\u}(E)\to H_{b\mathfrak B}(F)$ is a continuous multiplicative operator and define
 $g:F'' \to E''$ by $g(z)=\pi(\widetilde{\delta}_{z}\circ \phi)$. Then, $g$ is
 holomorphic and for every $\gamma\in E'$, $AB(\gamma)\circ g=AB(\phi \gamma)$. In particular, if the
 finite type polynomials are dense on $\u_k(E)$ (for every $k$), then $AB(\phi f)=AB(f)\circ g$ for every $f\in H_{b\u}(E)$.
\end{lemma}
\begin{proof}
Denote by $\theta_\phi:M_{b\mathfrak B}(F)\to M_{b\u}(E)$ the
restriction of the transpose of $\phi$. Then $g$ is just the
composition $F'' \overset{\widetilde{\delta}}{\longrightarrow}
M_{b\mathfrak B}(F) \overset{\theta_\phi}{\longrightarrow}
M_{b\u}(E) \overset{\pi}{\longrightarrow} E''$. If we take $z\in
F''$ and $\gamma\in E'$, then
$g(z)(\gamma)=\widetilde{\delta}_{z}(\phi \gamma)=AB(\phi
\gamma)(z)$. Thus $g$ is weak*-holomorphic on $F''$ and therefore
holomorphic (see for example \cite[Example 8D]{Muj86}).

If $\gamma\in E'$ then $AB(\gamma)(g(z))=g(z)(\gamma)= AB(\phi
\gamma)(z)$. Since $\phi$ multiplicative and continuous, the last
assertion follows.
\end{proof}

Although it is hard for a Banach space $E$ to satisfy that finite type polynomials are dense in $H_b(E)$ ($c_0$ and Tsirelson like spaces do, but no other classical Banach spaces), it is not so uncommon that finite type polynomials be dense in $H_{b\u}(E)$ for certain sequences $\u$ and Banach spaces $E$. Besides those sequences where finite type polynomials are automatically dense (such as approximable or nuclear polynomials), there are combination of ideals and Banach spaces that make finite type polynomials dense (see Example~\ref{ejemplo-Banach-Stone} below).

\begin{theorem}\label{Banach-Stone}
(a) Let $\u$ be an $AB$-closed multiplicative sequence such that
finite type polynomials are dense in $\u_k(E'')$. Then, $H_{b\u}(E)$
and $H_{b\u}(F)$ are topologically isomorphic algebras if and only if $E'$ and $F'$ are isomorphic.

(b) Let $\u$ and $\mathfrak B$ be multiplicative sequences,
$\mathfrak B$ also $AB$-closed. Suppose that finite type polynomials
are dense in $\u_k(E)$ and on $\mathfrak B_k(E'')$ for some Banach
space $E$, for all $k$. If $H_{b\u}(E)$ and $H_{b\mathfrak B}(F)$
are topologically isomorphic algebras, then $E'$ is isomorphic to
$F'$.
\end{theorem}
\begin{proof}
$(a)$ $(b)$ If $E'$ and $F'$ are isomorphic, so are $E''$ and $F''$. Therefore, finite type polynomials are also dense in $\mathfrak A_k(F'')$. Thus, $\u$ is regular both at $E$ and $F$ and
we can follow the reasoning in \cite{CarLas04,LasZal00} to obtain the desired isomorphism. The converse is a particular case of $(b)$.

$(b)$ Suppose that $\phi:H_{b\u}(E)\to H_{b\mathfrak B}(F)$ is an isomorphism. Let $g:F'' \to E''$ and $h:E''\to F''$ be the applications given by Lemma \ref{la g es holo} for $\phi$ and $\phi^{-1}$ respectively. Then $h\circ g$ is the composition
$$
  F''\overset{\widetilde{\delta}}{\to} M_{b\mathfrak B}(F)
  \overset{\theta_\phi}{\to} M_{b\u}(E) \overset{\pi}{\to}  E'' \overset{\widetilde{\delta}}{\to} M_{b\u}(E) \overset{\theta_{\phi^{-1}}}{\to} M_{b\mathfrak B}(F)\overset{\pi}{\to} F''.
$$
Since $M_{b\u}(E)=\widetilde{\delta}(E'')$, it follows that $h\circ
g=id_{F''}$. Thus $dh(g(0))\circ dg(0)=id_{F''}$ and therefore $F''$
is isomorphic to a complemented subspace of $E''$ which implies that
every polynomial in $\mathfrak B_k(F'')$ is approximable. Since
$\mathfrak B$ is $AB$-closed we can conclude that every polynomial
in $\mathfrak B_k(F)$ is approximable (if $P\in\mathfrak B_k(F)$
then $AB(P)\in\mathfrak B_k(F'')$, thus $AB(P)$ is approximable and
therefore $P$ is approximable). Now, since $M_{b\mathfrak
B}(F)=\widetilde{\delta}(F'')$, we can prove similarly that $g\circ
h=id_{E''}$, that is, $h=g^{-1}$, and differentiating at $g(0)$ we
obtain that $E''$ is isomorphic to $F''$.

Since every polynomial on $\mathfrak B_k(F)$ is approximable we have
that $\phi \gamma$ is weakly continuous on bounded sets for every $\gamma\in E'$ and then
$AB(\phi \gamma)$ is $w^*$-continuous on bounded sets. The identity
$g(z)(\gamma)=AB(\phi\gamma)(z)$ shown in Lemma~\ref{la g es holo}
assures then that $g$ is $w^*$-$w^*$-continuous on bounded sets.
Similarly, $g^{-1}$ is $w^*$-$w^*$-continuous on bounded sets.
Moreover, applying \cite[Lemma 2.1]{AroColGam95} to $z\mapsto
g(z)(\gamma)$, we obtain that de differential of $g$ at any point is
$w^*$-$w^*$-continuous (and analogously for $g^{-1}$). Therefore,
the isomorphism between $E''$ and $F''$ is the transpose of an
isomorphism between $F'$ and $E'$.
\end{proof}

\begin{example}\label{ejemplo-Banach-Stone}\rm
Finite type polynomials are dense in the space of integral polynomial on any Asplund Banach space, since in this case, integral and nuclear polynomials coincide (see \cite{Ale85,BoyRya01,CarDim00}). Moreover, finite type polynomials are also dense in the space of extendible polynomials on an Asplund space \cite{CarGal-2}.
Then, as a consequence of Theorem \ref{Banach-Stone} we have:  suppose that $E$ and $F$ are Banach spaces, one of them reflexive, and let $\u$ and $\mathfrak B$ can be any of the sequence of nuclear, integral, approximable or extendible polynomials. If $H_{b\u}(E)$ is isomorphic to
 $H_{b\mathfrak B}(F)$, then $F$ is isomorphic to $E$.
\end{example}


\begin{thebibliography}{10}

\bibitem{Ale85}
Raymundo Alencar.
\newblock Multilinear mappings of nuclear and integral type.
\newblock {\em Proc. Amer. Math. Soc.}, 94(1):33--38, 1985.

\bibitem{Aro79}
Richard~M. Aron.
\newblock Weakly uniformly continuous and weakly sequentially continuous entire
  functions.
\newblock In {\em Advances in holomorphy ({P}roc. {S}em. {U}niv. {F}ed. {R}io
  de {J}aneiro, {R}io de {J}aneiro, 1977)}, pages 47--66. North--Holland Math.
  Studies, 34. North-Holland, Amsterdam, 1979.

\bibitem{AroColGam95}
Richard~M. Aron, Brian~J. Cole, and Theodore~W. Gamelin.
\newblock Weak-star continuous analytic functions.
\newblock {\em Canad. J. Math.}, 47(4):673--683, 1995.

\bibitem{AroGalGarMae96}
Richard~M. Aron, Pablo Galindo, Domingo Garc\'{\i}a, and Manuel Maestre.
\newblock {Regularity and algebras of analytic functions in infinite
  dimensions.}
\newblock {\em Trans. Am. Math. Soc.}, 348(2):543--559, 1996.

\bibitem{AroHerVal83}
Richard~M. Aron, Carlos Herv{\'e}s, and Manuel Valdivia.
\newblock Weakly continuous mappings on {B}anach spaces.
\newblock {\em J. Funct. Anal.}, 52(2):189--204, 1983.

\bibitem{BoyLas08}
Christopher Boyd and Silvia Lassalle.
\newblock Decomposable symmetric mappings between infinite-dimensional spaces.
\newblock {\em Ark. Mat.}, 46(1):7--29, 2008.

\bibitem{BoyRya01}
Christopher Boyd and Raymond~A. Ryan.
\newblock Geometric theory of spaces of integral polynomials and symmetric
  tensor products.
\newblock {\em J. Funct. Anal.}, 179(1):18--42, 2001.

\bibitem{Car99}
Daniel Carando.
\newblock Extendible polynomials on {B}anach spaces.
\newblock {\em J. Math. Anal. Appl.}, 233(1):359--372, 1999.

\bibitem{Car01}
Daniel Carando.
\newblock Extendibility of polynomials and analytic functions on {$l\sb p$}.
\newblock {\em Studia Math.}, 145(1):63--73, 2001.

\bibitem{CarDim00}
Daniel Carando and Ver{\'o}nica Dimant.
\newblock Duality in spaces of nuclear and integral polynomials.
\newblock {\em J. Math. Anal. Appl.}, 241(1):107--121, 2000.

\bibitem{CarDimMur07}
Daniel Carando, Ver{\'o}nica Dimant, and Santiago Muro.
\newblock Hypercyclic convolution operators on {F}r\'echet spaces of analytic
  functions.
\newblock {\em J. Math. Anal. Appl.}, 336(2):1324--1340, 2007.

\bibitem{CarDimMur09}
Daniel Carando, Ver\'onica Dimant, and Santiago Muro.
\newblock {Coherent sequences of polynomial ideals on Banach spaces.}
\newblock {\em Math. Nachr.}, 282(8):1111--1133, 2009.

\bibitem{CarGal-3}
Daniel Carando and Daniel Galicer.
\newblock {Extending polynomials in maximal and minimal ideals}.
\newblock {\em Preprint}.

\bibitem{CarGal-4}
Daniel Carando and Daniel Galicer.
\newblock {Natural symmetric tensor norms}.
\newblock {\em Preprint}.

\bibitem{CarGal-2}
Daniel Carando and Daniel Galicer.
\newblock {The symmetric Radon-Nikodym property for tensor norms}.
\newblock {\em Preprint}.

\bibitem{CarGarMae05}
Daniel Carando, Domingo Garc\'{\i}a, and Manuel Maestre.
\newblock {Homomorphisms and composition operators on algebras of analytic
  functions of bounded type.}
\newblock {\em Adv. Math.}, 197(2):607--629, 2005.

\bibitem{CarLas04}
Daniel Carando and Silvia Lassalle.
\newblock {$E'$} and its relation with vector-valued functions on {$E$}.
\newblock {\em Ark. Mat.}, 42(2):283--300, 2004.

\bibitem{CarZal99}
Daniel Carando and Ignacio Zalduendo.
\newblock A {H}ahn-{B}anach theorem for integral polynomials.
\newblock {\em Proc. Amer. Math. Soc.}, 127(1):241--250, 1999.

\bibitem{DavGam89}
Alexander~M. Davie and Theodore~W. Gamelin.
\newblock A theorem on polynomial-star approximation.
\newblock {\em Proc. Amer. Math. Soc.}, 106(2):351--356, 1989.

\bibitem{DefFlo93}
Andreas Defant and Klaus Floret.
\newblock {\em {Tensor norms and operator ideals.}}
\newblock {North-Holland Mathematics Studies. 176. Amsterdam: North-Holland.
  xi, 566 p. }, 1993.

\bibitem{DimGalMaeZal04}
Ver{\'o}nica Dimant, Pablo Galindo, Manuel Maestre, and Ignacio Zalduendo.
\newblock Integral holomorphic functions.
\newblock {\em Studia Math.}, 160(1):83--99, 2004.

\bibitem{Din99}
Se\'an Dineen.
\newblock {\em {Complex analysis on infinite dimensional spaces.}}
\newblock {Springer Monographs in Mathematics. London: Springer.}

\bibitem{Din71(holomorphy-types)}
Se\'an Dineen.
\newblock {Holomorphy types on a Banach space.}
\newblock {\em Studia Math.}, 39:241--288, 1971.

\bibitem{Flo97}
Klaus Floret.
\newblock {Natural norms on symmetric tensor products of normed spaces.}
\newblock {\em Note Mat.}, 17:153--188, 1997.

\bibitem{Flo01}
Klaus Floret.
\newblock {Minimal ideals of $n$-homogeneous polynomials on Banach spaces.}
\newblock {\em Result. Math.}, 39(3-4):201--217, 2001.

\bibitem{Flo02}
Klaus Floret.
\newblock {On ideals of $n$-homogeneous polynomials on Banach spaces.}
\newblock {Strantzalos, P. (ed.) et al., Topological algebras with applications
  to differential geometry and mathematical physics. Proceedings of the
  Fest-Colloquium in honour of Professor A. Mallios, University of Athens,
  Athens, Greece, September 16--18, 1999. Athens: University of Athens,
  Department of Mathematics. 19-38 (2002).}, 2002.

\bibitem{FloHun02}
Klaus Floret and Stephan Hunfeld.
\newblock {Ultrastability of ideals of homogeneous polynomials and multilinear
  mappings on Banach spaces.}
\newblock {\em Proc. Am. Math. Soc.}, 130(5):1425--1435, 2002.

\bibitem{Gup70}
Chaitan~P. Gupta.
\newblock On the {M}algrange theorem for nuclearly entire functions of bounded
  type on a {B}anach space.
\newblock {\em Nederl. Akad. Wetensch. Proc. Ser. A73 = Indag. Math.},
  32:356--358, 1970.

\bibitem{Har97}
Lawrence~A. Harris.
\newblock {A Bernstein-Markov theorem for normed spaces.}
\newblock {\em J. Math. Anal. Appl.}, 208(2):476--486, 1997.

\bibitem{LasZal00}
Silvia Lassalle and Ignacio Zalduendo.
\newblock To what extent does the dual {B}anach space {$E'$} determine the
  polynomials over {$E$}?
\newblock {\em Ark. Mat.}, 38(2):343--354, 2000.

\bibitem{Mor84}
Luiza~A. Moraes.
\newblock The {H}ahn-{B}anach extension theorem for some spaces of
  {$n$}-homogeneous polynomials.
\newblock In {\em Functional analysis: surveys and recent results, {III}
  ({P}aderborn, 1983)}, volume~90 of {\em North-Holland Math. Stud.}, pages
  265--274. North-Holland, Amsterdam, 1984.

\bibitem{Muj86}
Jorge Mujica.
\newblock {\em {Complex analysis in Banach spaces. Holomorphic functions and
  domains of holomorphy in finite and infinite dimensions.}}
\newblock {North-Holland Mathematics Studies, 120. Notas de Matem\'atica, 107.
  Amsterdam/New York/Oxford: North-Holland. XI. }, 1986.

\bibitem{Nac69}
Leopoldo Nachbin.
\newblock {\em Topology on spaces of holomorphic mappings}.
\newblock Ergebnisse der Mathematik und ihrer Grenzgebiete, Band 47.
  Springer-Verlag New York Inc., New York, 1969.

\end{thebibliography}
\end{document}